\documentclass{amsart}

\usepackage{amsmath, amsthm, amssymb}
\usepackage{tikz}
\usetikzlibrary{decorations.pathreplacing}
\usepackage{array}

\newcommand{\snort}{\textsc{Snort }}
\newcommand{\col}{\textsc{Col }}
\newcommand{\nogo}{\textsc{NoGo }}

\newcommand{\Snort}{\textsc{Snort}}
\newcommand{\Col}{\textsc{Col}}
\newcommand{\Nogo}{\textsc{NoGo}}

\newcommand{\SRideal}[1]{\mathcal{N}(\Delta_{#1})}
\newcommand{\facetI}[1]{\mathcal{F}(\Delta_{#1})}
\newcommand{\SRsc}[1]{\mathcal{N}(I_{#1})}
\newcommand{\facetsc}[1]{\mathcal{F}(I_{#1})}
\newcommand{\facet}[1]{\mathcal{F}(#1)}

\newcommand{\legalI}[1]{\mathcal{L}_{#1}}
\newcommand{\illegalI}[1]{\mathcal{ILL}_{#1}}
\newcommand{\auxboard}[1]{\Gamma_{#1}}

\theoremstyle{definition} \newtheorem{definition}{Definition}[section]
\theoremstyle{plain} \newtheorem{theorem}[definition]{Theorem}
\theoremstyle{plain} 
\theoremstyle{plain} \newtheorem{proposition}[definition]{Proposition}
\theoremstyle{plain} 
\theoremstyle{plain} 
\theoremstyle{definition} \newtheorem{example}[definition]{Example}
\theoremstyle{remark} 
\theoremstyle{definition} \newtheorem*{ruleset1}{Illegal Ruleset}
\theoremstyle{definition} \newtheorem*{ruleset2}{Legal Ruleset}

\usepackage{xspace}
\makeatletter
\DeclareRobustCommand\onedot{\futurelet\@let@token\@onedot}
\def\@onedot{\ifx\@let@token.\else.\null\fi\xspace}
 
\def\ie{{i.e}\onedot}

\def\etal{{et al}\onedot}
\makeatother

\begin{document}

\title{Games and Complexes I: Transformation via Ideals}
\author{Sara Faridi \and Svenja Huntemann \and Richard J.~Nowakowski}

\address{Dalhousie University\\
Halifax, Nova Scotia\\
B3H 3J5, Canada}
\email{faridi@mathstat.dal.ca, svenjah@mathstat.dal.ca, rjn@mathstat.dal.ca}


\thanks{The first and third authors' research was supported in part by the Natural Sciences and Engineering Research Council of Canada.}

\begin{abstract}
Placement games are a subclass of combinatorial games which are played on graphs. We will demonstrate that one can construct simplicial complexes corresponding to a placement game, and this game could be considered as a game played on these simplicial complexes. These complexes are constructed using square-free monomials.
\end{abstract}

\maketitle

\section{Introduction}
We will demonstrate a relationship between a subclass of combinatorial games, such as \textsc{Domineering} and \Col, and algebraic structures defined on simplicial complexes. There are two relationships, one via the maximal legal positions and the other through the minimal illegal positions. We will begin by giving the necessary background, first from combinatorial game theory, then from combinatorial commutative algebra.

For a game \textit{perfect information} means that both players know which game they are playing, on which board, and the current position. No \textit{chance} means that no dice can be rolled or cards can be dealt, or any other item involving probability can be used.

\begin{definition}\label{def:game}
A \textit{combinatorial game} is a 2-player game with perfect information and no chance, where the two players are \textit{Left} and \textit{Right} (denoted by $L$ and $R$ respectively) and they do not move simultaneously. Then a game is a set $P$ of \textit{positions} with a specified starting position. \textit{Rules} determine from which position to which position the players are allowed to move. A \textit{legal position} is a position that can be reached by playing the game from the starting position (which is legal) according to the rules. Moving from position $P$ to position $Q$ is called a \textit{legal move} if both $P$ and $Q$ are legal positions and the move is allowed according to the rules. $Q$ is usually called an option of $P$.
\end{definition}

In this paper, a combinatorial game will be denoted by its name in \textsc{Small Caps}. Well known examples of combinatorial games are \textsc{Chess}, \textsc{Checkers}, \textsc{Tic-Tac-Toe}, \textsc{Go}, and \textsc{Connect Four}. Examples of games that are not combinatorial games include bridge, backgammon, poker, and Snakes and Ladders.

Although games usually have a `winning condition' associated to them, \ie rules as to which player wins, for the purposes of this paper games do not need to have a notion of winning identified.

We will assume that the board on which games are played is a graph (or can be represented as a graph). A space on a board is then equivalent to a vertex and we use the two terms interchangeably.

\begin{definition}
A \textit{strong placement game} is a combinatorial game which satisfies the following:
\begin{itemize}
\item[(i)] The starting position is the empty board.
\item[(ii)] Players place pieces on empty spaces of the board according to the rules.
\item[(iii)] Pieces are not moved or removed once placed.
\item[(iv)] The rules are such that if it is possible to reach a position through a sequence of legal moves, then any sequence of moves leading to this position consists only of legal moves.
\end{itemize}
The \textsc{Trivial} placement game on a board is the strong placement game that has no additional rules.
\end{definition}

A \textit{basic position} is a board with only one piece placed. Any position, whether legal or illegal, in a strong placement game can be decomposed into basic positions.

The concept of a placement game originates in Brown et al \cite{GamePol} where condition (iv) is replaced by the condition that if it is legal to place a piece at one point, it must have been legal at any point before. We call this type of game a `medium placement game'. A `weak placement game' is a combinatorial game that satisfies the above conditions (i) through (iii).

Note that (iv) implies that every subposition of a legal position is also legal.

Placement games were only recently defined formally by Brown \etal in \cite{GamePol}, even though several placement games, for example \textsc{Tic-Tac-Toe} or \textsc{Domineering}, have been known and studied for a long time. In this work, we will consider strong placement games exclusively.

Throughout this paper `placement game' refers to a strong placement game. 

Here are three more we will use as examples.

\begin{definition}
In \Snort, players may not place pieces on a vertex adjacent to a vertex containing a piece from their opponent.
\end{definition}

\begin{definition}\label{def:col}
In \Col, players may not place pieces on a vertex adjacent to a vertex containing one of their own pieces.
\end{definition}

\begin{definition}
In \Nogo, at every point in the game, for each maximal group of connected vertices of the board that contain pieces placed by the same player, one of these needs to be adjacent to an empty vertex.
\end{definition}

In these games, the pieces only occupy one vertex each, which is in fact not necessary. For example in \textsc{Crosscram} \cite{Ga74} and \textsc{Domineering} \cite{BCG04} the players' pieces occupy two adjacent vertices.

\begin{definition}
The \textit{disjunctive sum} between two positions of combinatorial games $G$ and $H$ is the position in which a player can play in one of $G$ and $H$ but not both simultaneously.
\end{definition}

Assuming implicitly that placement games are part of a disjunctive sum implies that a board might be filled with more pieces of one player than of the other. Making this assumption is very useful since in many placement games the board might `break up' into the disjunctive sum of smaller boards. 

\begin{example}
For an example, consider \col played on the path $P_7$. Then the position on the left of Figure \ref{fig:disjunctive3} is equivalent to the one in which the middle space is `deleted' (on the right), \ie it is equivalent to the disjunctive sum of the two \col positions on the right, one of which has two Right pieces but no Left pieces.
\begin{figure}[!ht]
	\begin{center}
	\begin{tikzpicture}[scale=0.5]
		\foreach \y in {0,1}{
			\draw (0,\y)--(7,\y);
			\draw (8,\y)--(11,\y);
			\draw (12,\y)--(15,\y);}
		\foreach \x in {0,1,2,3,4,5,6,7,8,9,10,11,12,13,14,15}{
			\draw (\x,0)--(\x,1);}
		\node at (0.5, 0.5) {$R$};
		\node at (2.5, 0.5) {$R$};
		\node at (3.5, 0.5) {$L$};
		\node at (4.5, 0.5) {$R$};
		\node at (6.5, 0.5) {$L$};
		\node at (8.5, 0.5) {$R$};
		\node at (10.5, 0.5) {$R$};
		\node at (12.5, 0.5) {$R$};
		\node at (14.5, 0.5) {$L$};
		\node at (7.5, 0.5) {$\cong$};
		\node at (11.5, 0.5) {$+$};
	\end{tikzpicture}
	\end{center}
	\caption{A \col Position That is the Disjunctive Sum of Two \col Positions}
	\label{fig:disjunctive3}
\end{figure}
\end{example}

For a placement game $G$ and a board $B$, let \[f_i(G, B)\] denote the number of positions with $i$ pieces played, regardless of which player the pieces belong to. If the game and board are clear from context, we shorten the notation to $f_i$.

\begin{definition}[Brown \etal \cite{GamePol}]\label{def:gamepol}
For a game $G$ played on a board $B$, the \textit{game polynomial} is defined to be
\[P_{G, B}(x)=\sum_{i=0}^k f_i(G,B)x^i.\]
$P_{G,B}(1)$ is then the total number of legal positions of the game.
\end{definition}

The motivation for game polynomials came from Farr \cite{Fa03} in 2003 where the number of end positions and some polynomials of the game \textsc{Go} were considered, and work in this area was continued by Tromp and Farneb\"ack \cite{TF07} in 2007 and by Farr and Schmidt \cite{FS08} in 2008. Even though \textsc{Go} is not a placement game since pieces are removed, it shares many properties with this class of games. Thus it was natural for the authors of \cite{GamePol} to consider the concept of game polynomials for placement games.

We will now introduce concepts from combinatorial commutative algebra that we will need to construct simplicial complexes equivalent to placement games.

\begin{definition}
A \textit{simplicial complex} $\Delta$ on a finite vertex set $V$ is a set of subsets (called \textit{faces}) of $V$ with the conditions that if $A\in \Delta$ and $B\subseteq A$, then $B\in \Delta$. The \textit{facets} of a simplicial complex $\Delta$ are the maximal faces of $\Delta$ with respect to inclusion. A \textit{non-face} of a simplicial complex $\Delta$ is a subset of its vertices that is not a face. The \textit{$f$-vector} $(f_0, f_1, \ldots, f_k)$ of a simplicial complex $\Delta$ enumerates the number of faces $f_i$ with $i$ vertices. Note that if $\Delta\neq\emptyset$, then $f_0=1$.
\end{definition}

In the algebraic literature, the $f$-vector of a complex is usually indexed from $-1$ to $k-1$ as this is the ``dimension'' of the face (the number of vertices minus 1). Due to the connection between placement games and simplicial complexes, we have chosen the combinatorial indexing.

Recall that an \textit{ideal} $I$ of a ring $R=R(+,\cdot)$ is a subset of $R$ such that $(I,+)$ is a subgroup of $R$ and $rI\subseteq I$ for all $r\in R$.

Let $k$ be a field and $R=k[x_1,\ldots,x_n]$ a polynomial ring. Given a simplicial complex $\Delta$ on $n$ vertices, we can label each vertex with an integer from $1$ to $n$. Each face $F$ (resp.~non-face $N$) of $\Delta$ can then be represented by a square-free monomial of $R$ by including $x_i$ in the monomial representing the face $F$ (resp.~the non-face $N$) if and only if the vertex $i$ belongs to $F$ (resp.~$N$). We then have the following (see \cite{BH93} and \cite{Fa02} for more information):

\begin{definition}\label{def:SRFideal}
The \textit{facet ideal} of a simplicial complex $\Delta$, denoted by $\facetI{}$, is the ideal generated by the monomials representing the facets of $\Delta$. The \textit{Stanley-Reisner ideal} of a simplicial complex $\Delta$, denoted by $\SRideal{}$, is the ideal generated by the monomials representing the minimal non-faces of $\Delta$.
\end{definition}

\begin{definition}\label{def:SRFcomplex}
The \textit{facet complex} of a square-free monomial ideal $I$, denoted by $\facetsc{}$, is the simplicial complex whose facets are represented by the square-free monomials generating $I$. The \textit{Stanley-Reisner complex} of a square-free monomial ideal $I$, denoted by $\SRsc{}$, is the simplicial complex whose faces are represented by the square-free monomials not in $I$.
\end{definition}

To clarify these concepts, we will give two examples:
\begin{example}
Consider the simplicial complex $\Delta$ in Figure \ref{fig:SRfacetex1} with the labeling of the vertices as given.

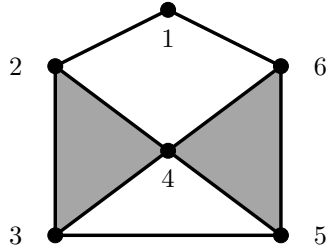
\begin{figure}[!ht]
\begin{center}
\begin{tikzpicture}[scale=1.5]
	\draw[line width=1.3] (0,0)--(2,0);
	\draw[line width=1.3] (0,1.5)--(1,2)--(2,1.5);
	\filldraw[fill=gray!70, line width=1.3] (0,1.5)--(0,0)--(1,0.75)--cycle;
	\filldraw[fill=gray!70, line width=1.3] (2,1.5)--(2,0)--(1,0.75)--cycle;
	\filldraw (0,0) circle (0.066cm);
	\filldraw (2,0) circle (0.066cm);
	\filldraw (0,1.5) circle (0.066cm);
	\filldraw (1,2) circle (0.066cm);
	\filldraw (2,1.5) circle (0.066cm);
	\filldraw (1,0.75) circle (0.066cm);
	\draw (-0.35,0) node {$3$};
	\draw (-0.35,1.5) node {$2$};
	\draw (2.35,0) node {$5$};
	\draw (2.35,1.5) node {$6$};
	\draw (1,0.5) node {$4$};
	\draw (1,1.75) node {$1$};
\end{tikzpicture}
\end{center}
\caption{An Example of a Simplicial Complex}
\label{fig:SRfacetex1}
\end{figure}

The facet ideal of $\Delta$ then is
\[\facetI{}=\langle x_1x_2, x_1x_6, x_2x_3x_4, x_3x_5, x_4x_5x_6\rangle,\]
and the Stanley-Reisner ideal of $\Delta$ is
\[\SRideal{}=\langle x_1x_3, x_1x_4, x_1x_5, x_2x_5, x_2x_6, x_3x_4x_5, x_3x_6\rangle.\]
\end{example}

\begin{example}
Consider the square-free monomial ideal $I=\langle x_1x_3, x_2x_3x_4\rangle$. The facet complex $\facetsc{}$ is given in Figure \ref{fig:SRfacetex2f} and the Stanley-Reisner complex $\SRsc{}$ is given in Figure \ref{fig:SRfacetex2SR}.

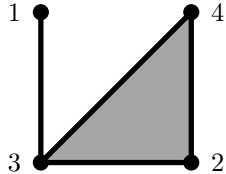
\begin{figure}[!ht]
\begin{center}
\begin{tikzpicture}
	\draw[line width=2] (0,0)--(0,2);
	\filldraw[fill=gray!70, line width=2] (0,0)--(2,0)--(2,2)--cycle;
	\filldraw (0,0) circle (0.1cm);
	\filldraw (2,0) circle (0.1cm);
	\filldraw (0,2) circle (0.1cm);
	\filldraw (2,2) circle (0.1cm);
	\draw (-0.35,0) node {$3$};
	\draw (-0.35,2) node {$1$};
	\draw (2.35,0) node {$2$};
	\draw (2.35,2) node {$4$};
\end{tikzpicture}
\end{center}
\caption{Facet Complex of $I=\langle x_1x_3, x_2x_3x_4\rangle$}
\label{fig:SRfacetex2f}
\end{figure}

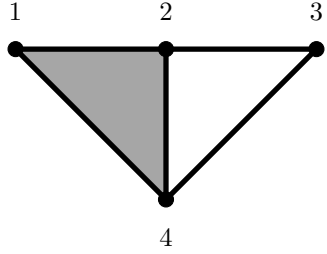
\begin{figure}[!ht]
\begin{center}
\begin{tikzpicture}
	\filldraw[fill=gray!70, line width=2] (-2,0)--(0,0)--(0,-2)--cycle;
	\draw[line width=2] (0,0)--(2,0);
	\draw[line width=2] (2,0)--(0,-2);
	\filldraw (0,0) circle (0.1cm);
	\filldraw (2,0) circle (0.1cm);
	\filldraw (0,-2) circle (0.1cm);
	\filldraw (-2,0) circle (0.1cm);
	\draw (-2,0.5) node {$1$};
	\draw (0,0.5) node {$2$};
	\draw (2,0.5) node {$3$};
	\draw (0,-2.5) node {$4$};
\end{tikzpicture}
\end{center}
\caption{Stanley-Reisner Complex of $I=\langle x_1x_3, x_2x_3x_4\rangle$}
\label{fig:SRfacetex2SR}
\end{figure}
\end{example}

It is clear that the facet operators are inverses of each other, i.e. $\mathcal{F}(\facetI{})=\Delta$ and $\mathcal{F}(\facetsc{})=I$, from their definitions. This is also true of the Stanley-Reisner operators: A minimal non-face of $\SRsc{}$ is a minimal monomial generator of $I$, thus a generator of $I$, showing $\mathcal{N}(\SRsc{})=I$. Similarly, since $\SRideal{}$ contains all monomials representing non-faces, a square-free monomial not in $\SRideal{}$ has to be a face of $\Delta$, thus $\mathcal{N}(\SRideal{})=\Delta$.

This shows that both the facet and the Stanley-Reisner operators give a bijection between the set of all square-free monomial ideals in $n$ variables and the set of all simplicial complexes on $n$ vertices.

\section{Constructing Monomials and Simplicial Complexes from Placement Games}
We will now introduce a construction that associates a set of monomials and a simplicial complex to each placement game.

Given a placement game $G$ on a board $B$, we can construct a set of square-free monomials in the following way: First, label the basic positions by $1,2,\ldots,n$. For each legal position we then create a square-free monomial by including $x_i$ if Left has played in position $i$ and $y_j$ if Right has placed in position $j$. The empty position (before anyone has started playing) is represented by $1$.

\begin{example}\label{ex:Col2}
Consider \col played on the path $P_3$. We label the basic positions, in this case the spaces of the board, as given in Figure \ref{fig:ColP3label}.

\begin{figure}[!ht]
	\begin{center}
	\begin{tikzpicture}[scale=0.5]
		\draw (0, 0) -- (3, 0) -- (3, -1) -- (0, -1) -- (0, 0);
		\draw (1, 0) -- (1, -1);
		\draw (2, 0) -- (2, -1);
		\draw (0.5, -0.5) node {1};
		\draw (1.5, -0.5) node {2};
		\draw (2.5, -0.5) node {3};
	\end{tikzpicture}
	\end{center}
	\caption{Labeling $P_3$}
	\label{fig:ColP3label}
\end{figure}
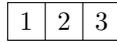

The maximum legal positions and their corresponding monomials are given in Figure \ref{fig:ColP3maxpos}.

\begin{figure}[!ht]
	\begin{center}
	\begin{tikzpicture}[scale=0.5]
		\draw (0, 0) -- (3, 0) -- (3, -1) -- (0, -1) -- (0, 0);
		\draw (1, 0) -- (1, -1);
		\draw (2, 0) -- (2, -1);
		\draw (0.5, -0.5) node {$L$};
		\draw (1.5, -0.5) node {$R$};
		\draw (2.5, -0.5) node {$L$};
		\draw (5,-0.5) node {$x_1y_2x_3$};
		\draw (8, 0) -- (11, 0) -- (11, -1) -- (8, -1) -- (8, 0);
		\draw (9, 0) -- (9, -1);
		\draw (10, 0) -- (10, -1);
		\draw (8.5, -0.5) node {$L$};
		\draw (10.5, -0.5) node {$R$};
		\draw (13,-0.5) node {$x_1y_3$};
		\draw (0, -2) -- (3, -2) -- (3, -3) -- (0, -3) -- (0, -2);
		\draw (1, -2) -- (1, -3);
		\draw (2, -2) -- (2, -3);
		\draw (0.5, -2.5) node {$R$};
		\draw (1.5, -2.5) node {$L$};
		\draw (2.5, -2.5) node {$R$};
		\draw (5,-2.5) node {$y_1x_2y_3$};
		\draw (8, -2) -- (11, -2) -- (11, -3) -- (8, -3) -- (8, -2);
		\draw (9, -2) -- (9, -3);
		\draw (10, -2) -- (10, -3);
		\draw (8.5, -2.5) node {$R$};
		\draw (10.5, -2.5) node {$L$};
		\draw (13,-2.5) node {$y_1x_3$};
	\end{tikzpicture}
	\end{center}
	\caption{Maximum Legal Positions for \col on $P_3$}
	\label{fig:ColP3maxpos}
\end{figure}
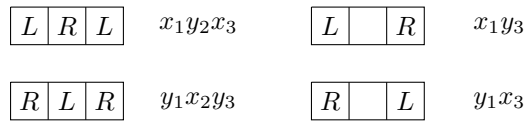
\end{example}

Using these monomials, we can build a simplicial complex $\Delta_{G,B}$ on the vertex set $V=\{x_1, \ldots, x_n, y_1, \ldots, y_o\}$ by letting a subset $F$ of $V$ be a face if and only if there exists a square-free monomial $m$ representing a legal position such that each element of $F$ divides $m$.

\begin{definition}\label{def:gamecomplex}
A simplicial complex that can be constructed from a placement game $G$ on a board $B$ in this way is called a \textit{legal complex} and is denoted by $\Delta_{G,B}$.
\end{definition}

\begin{example}\label{ex:ColP3}
Consider \col played on the path $P_3$. Using the notation from Example \ref{ex:Col2}, we get the legal complex $\Delta_{\Col, P_3}$ as given in Figure \ref{fig:ColP3sc}.

\begin{figure}[!ht]
	\begin{center}
	\begin{tikzpicture}[scale=1]
		\filldraw[fill=gray!70, line width=2pt, draw=black] (0, 0)--(1.73, 1)--(1.73, -1)-- cycle;
		\filldraw[fill=gray!70, line width=2pt, draw=black] (5.46, 0)--(3.73, 1)--(3.73, -1)-- cycle;
		\draw[line width=2pt] (1.73, 1)--(3.73, 1);
		\draw[line width=2pt] (1.73, -1)--(3.73, -1);
		\filldraw (0,0) circle (0.1cm);
		\filldraw (1.73,1) circle (0.1cm);
		\filldraw (1.73,-1) circle (0.1cm);
		\filldraw (5.46,0) circle (0.1cm);
		\filldraw (3.73,1) circle (0.1cm);
		\filldraw (3.73,-1) circle (0.1cm);
		\draw (-0.7,0) node {$y_2$};
		\draw (1.73,1.4) node {$x_1$};
		\draw (1.73,-1.5) node {$x_3$};
		\draw (5.98,0) node {$x_2$};
		\draw (3.73,1.4) node {$y_3$};
		\draw (3.73,-1.5) node {$y_1$};
	\end{tikzpicture}
	\end{center}
	\caption{The Legal Complex $\Delta_{\Col, P_3}$}
	\label{fig:ColP3sc}
\end{figure}
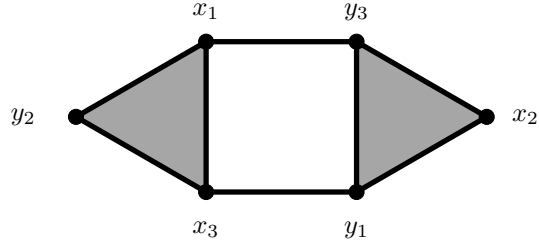
\end{example}

Observe that the maximum legal positions of a game, \ie the positions in which no piece can be placed by either Left or Right (so the game ends), correspond to the facets of $\Delta_{G,B}$ and thus uniquely determine $\Delta_{G,B}$.

In game theoretic terms, the $f$-vector of a legal complex $\Delta_{G,B}$ indicates that there are $f_i$ legal positions with $i$ pieces in the game $G$, regardless if pieces belong to Left or to Right. Thus for placement games the entries of the $f$-vector of the legal complex $\Delta_{G,B}$ are the coefficients of the game polynomial $P_{G, B}$. Therefore we have the following:
\begin{proposition}
\begin{align*}
	f_i(G,B)&=\text{number of legal positions in }G\text{ with }i\text{ pieces played on }B,\\
	&=\text{number of degree }i\text{ monomials representing legal positions in }G,\\
	&=\text{number of faces with }i\text{ vertices in }\Delta_{G,B},
\end{align*}
and we can use any of these concepts to find $f_i$.
\end{proposition}
This also justifies using the same notation for the coefficients of a game polynomial as for entries of a $f$-vector.

We now give three more examples for the construction of monomials and simplicial complexes.

\begin{example}
The cycle $C_3$ is labelled as in Figure \ref{fig:ColC3label}.
            
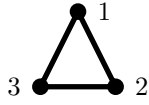
\begin{figure}[!ht]
	\begin{center}
	\begin{tikzpicture}[scale=0.5]
		\draw[line width=2pt, fill=black] (0,0) circle (0.15cm) -- (1,-2) circle (0.15cm) -- (-1,-2) circle (0.15cm)-- (0,0);
		\draw (0.7, 0) node {1};
		\draw (1.7, -2) node {2};
		\draw (-1.7, -2) node {3};
	\end{tikzpicture}
	\end{center}
	\caption{Labeling $C_3$}
	\label{fig:ColC3label}
\end{figure}

Now consider \col on $C_3$. The monomials corresponding to the maximum legal positions are 
\[\{x_1y_2, x_1y_3, x_2y_3, y_1x_2, y_1x_3, y_2x_3\}.\]

Also consider \snort played on $P_3$ and $C_3$. The maximum monomials then are \[\{x_1x_2x_3, y_1y_2y_3, x_1y_3, x_3y_1\}\]
and \[\{x_1x_2x_3, y_1y_2y_3\}\] respectively.

The legal complexes of all three games are given in Figure \ref{fig:ColSnortP3C3sc}.

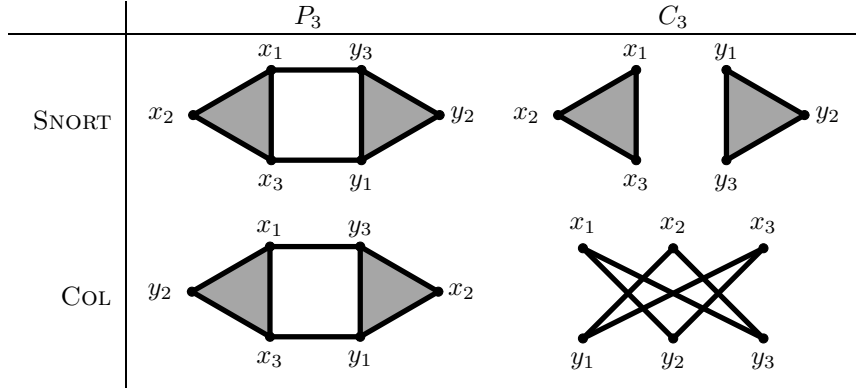
\begin{figure}[!ht]
	\begin{center}
	\begin{tabular}{>{\flushright\arraybackslash} m{1.2cm}| >{\centering\arraybackslash} m{4.5cm} >{\centering\arraybackslash} m{4.5cm}}
		& $P_3$ & $C_3$\\\hline
		\snort &
		\begin{tikzpicture}[scale=0.6]
		\filldraw[fill=gray!70, line width=2pt, draw=black] (0, 0)--(1.73, 1)--(1.73, -1)-- cycle;
		\filldraw[fill=gray!70, line width=2pt, draw=black] (5.46, 0)--(3.73, 1)--(3.73, -1)-- cycle;
		\draw[line width=2pt] (1.73, 1)--(3.73, 1);
		\draw[line width=2pt] (1.73, -1)--(3.73, -1);
		\filldraw (0,0) circle (0.1cm);
		\filldraw (1.73,1) circle (0.1cm);
		\filldraw (1.73,-1) circle (0.1cm);
		\filldraw (5.46,0) circle (0.1cm);
		\filldraw (3.73,1) circle (0.1cm);
		\filldraw (3.73,-1) circle (0.1cm);
		\draw (-0.7,0) node {$x_2$};
		\draw (1.73,1.4) node {$x_1$};
		\draw (1.73,-1.5) node {$x_3$};
		\draw (5.98,0) node {$y_2$};
		\draw (3.73,1.4) node {$y_3$};
		\draw (3.73,-1.5) node {$y_1$};
		\end{tikzpicture} &
		\begin{tikzpicture}[scale=0.6]
		\filldraw[fill=gray!70, line width=2pt, draw=black] (0, 0)--(1.73, 1)--(1.73, -1)-- cycle;
		\filldraw[fill=gray!70, line width=2pt, draw=black] (5.46, 0)--(3.73, 1)--(3.73, -1)-- cycle;
		\filldraw (0,0) circle (0.1cm);
		\filldraw (1.73,1) circle (0.1cm);
		\filldraw (1.73,-1) circle (0.1cm);
		\filldraw (5.46,0) circle (0.1cm);
		\filldraw (3.73,1) circle (0.1cm);
		\filldraw (3.73,-1) circle (0.1cm);
		\draw (-0.7,0) node {$x_2$};
		\draw (1.73,1.4) node {$x_1$};
		\draw (1.73,-1.5) node {$x_3$};
		\draw (5.98,0) node {$y_2$};
		\draw (3.73,1.4) node {$y_1$};
		\draw (3.73,-1.5) node {$y_3$};
		\end{tikzpicture} \\
		\col &
		\begin{tikzpicture}[scale=0.6]
		\filldraw[fill=gray!70, line width=2pt, draw=black] (0, 0)--(1.73, 1)--(1.73, -1)-- cycle;
		\filldraw[fill=gray!70, line width=2pt, draw=black] (5.46, 0)--(3.73, 1)--(3.73, -1)-- cycle;
		\draw[line width=2pt] (1.73, 1)--(3.73, 1);
		\draw[line width=2pt] (1.73, -1)--(3.73, -1);
		\filldraw (0,0) circle (0.1cm);
		\filldraw (1.73,1) circle (0.1cm);
		\filldraw (1.73,-1) circle (0.1cm);
		\filldraw (5.46,0) circle (0.1cm);
		\filldraw (3.73,1) circle (0.1cm);
		\filldraw (3.73,-1) circle (0.1cm);
		\draw (-0.7,0) node {$y_2$};
		\draw (1.73,1.4) node {$x_1$};
		\draw (1.73,-1.5) node {$x_3$};
		\draw (5.98,0) node {$x_2$};
		\draw (3.73,1.4) node {$y_3$};
		\draw (3.73,-1.5) node {$y_1$};
		\end{tikzpicture} & 
		\begin{tikzpicture}[scale=0.6]
		\draw[line width=2pt] (0,0)--(2,2);
		\draw[line width=2pt] (2,2)--(4,0);
		\draw[line width=2pt] (4,0)--(0,2);
		\draw[line width=2pt] (0,2)--(2,0);
		\draw[line width=2pt] (2,0)--(4,2);
		\draw[line width=2pt] (4,2)--(0,0);
		\filldraw (0,0) circle (0.1cm);
		\filldraw (2,0) circle (0.1cm);
		\filldraw (4,0) circle (0.1cm);
		\filldraw (0,2) circle (0.1cm);
		\filldraw (2,2) circle (0.1cm);
		\filldraw (4,2) circle (0.1cm);
		\draw (0,-0.5) node {$y_1$};
		\draw (2,-0.5) node {$y_2$};
		\draw (4,-0.5) node {$y_3$};
		\draw (0,2.5) node {$x_1$};
		\draw (2,2.5) node {$x_2$};
		\draw (4,2.5) node {$x_3$};
		\end{tikzpicture}\\	
	\end{tabular}
	\end{center}
	\caption{The Legal Complexes $\Delta_{\Snort, P_3}$, $\Delta_{\Snort, C_3}$, $\Delta_{\Col, P_3}$, and $\Delta_{\Col, C_3}$}
	\label{fig:ColSnortP3C3sc}
\end{figure}

\end{example}

Note that the legal complexes of \col and \snort on $P_3$ are isomorphic. This is true whenever \col and \snort are played on a bipartite graph, see \cite{MSc}.

\section{The Ideals of a Placement Game}\label{sec:ideals}

Through the monomials that represent legal or illegal positions of a game, we can also associate square-free monomial ideals with a placement game.

\begin{definition}\label{def:legalI}
The \textit{legal ideal}, $\legalI{G,B}$, of a placement game $G$ played on the board $B$ is the ideal generated by the monomials representing maximal legal positions of $G$.
\end{definition}

\begin{definition}\label{def:illegalI}
The \textit{illegal ideal}, $\illegalI{G,B}$, of a placement game $G$ played on the board $B$ is the ideal generated by the monomials representing minimal illegal positions of $G$.
\end{definition}

\begin{definition}\label{def:auxboard}
The \textit{illegal complex}, sometimes called the \textit{auxiliary board} \cite{GamePol}, of a placement game $G$ played on the board $B$, is the simplicial complex whose facets are represented by the monomials of the minimal illegal positions of $G$. It is denoted by $\auxboard{G,B}$.
\end{definition}

The authors in \cite{GamePol} introduce the auxiliary board for ``independence placement games'', which is the class of placement games for which the illegal complex is a graph. The term `independence game' was chosen since the independence sets of $\Gamma_{G,B}$ (considered as a graph) correspond to the legal positions of $G$ played on $B$, \ie the faces of $\Delta_{G,B}$.

\begin{proposition}\label{thm:legalillegal}
For a placement game $G$ played on a board $B$ we have the following
\begin{itemize}
	\item[(1)] $\legalI{G,B}=\facetI{G,B},$
	\item[(2)] $\illegalI{G,B}=\facet{\auxboard{G,B}}=\SRideal{G,B}.$
\end{itemize}
\end{proposition}
\begin{proof}
(1) The facets of $\Delta_{G,B}$ represent the maximal legal positions of $G$. Thus $\facetI{G,B}$ is the ideal generated by the monomials representing the maximal legal positions, which is $\legalI{G,B}$ by definition.

(2) The facets of $\auxboard{G,B}$ are represented by the monomials of the minimal illegal positions of $G$, which by definition generate $\illegalI{G,B}$, proving the first equality.

Since the faces of $\Delta_{G,B}$ represent the legal positions of $G$, the minimal non-faces of $\Delta_{G,B}$ represent the minimal illegal positions, which generate $\illegalI{G,B}$. Thus $\illegalI{G,B}=\SRideal{G,B}$.
\end{proof}

\begin{example}
	Consider \col played on the path $P_3$ with labels as in Example \ref{ex:Col2}. We then have the legal ideal
	 \[\legalI{\Col, P_3}=\langle x_1y_2x_3, y_1x_2y_3, x_1y_3, y_1x_3\rangle\]
	 and the illegal ideal 
	 \[\illegalI{\Col, P_3}=\langle x_1x_2, x_2x_3, y_1y_2, y_2y_3\rangle.\]
	The illegal complex $\auxboard{\Col, P_3}$ is given in Figure \ref{fig:ColP3Aux}.
	
	\begin{figure}[!ht]
	\begin{center}
	\begin{tikzpicture}[scale=2]
		\node (1) at (0,0) {$x_1$};
		\node (2) at (1,0) {$x_2$};
		\node (3) at (2,0) {$x_3$};
		\node (4) at (0,-1) {$y_1$};
		\node (5) at (1,-1) {$y_2$};
		\node (6) at (2,-1) {$y_3$};
		\draw (1)--(2)--(3)--(6)--(5)--(4)--(1);
		\draw (2)--(5);
	\end{tikzpicture}
	\end{center}
	\caption{The Illegal Complex $\auxboard{\Col, P_3}$}
	\label{fig:ColP3Aux}
\end{figure}
\end{example}

\section{Playing Games on Simplicial Complexes}
In this section we show that games can be played on the illegal or legal complex rather than the board.

Since the facets of the illegal complex represent the minimal illegal positions, we can play on $\auxboard{G,B}$, instead of playing $G$ on the board $B$, according to the following rules:
\begin{ruleset1}\label{rules:illegal}

\begin{enumerate}
\item Left may only play on vertices labelled $x_i$, while Right may only play on vertices labelled $y_i$.
\item Given a facet, pieces played may not occupy all the vertices of the facet.
\end{enumerate}
\end{ruleset1}

Since the facets of $\auxboard{G,B}$ are the minimal illegal positions, any vertex set that does not contain all the vertices of any facet is a legal position of $G$. Thus playing on $\auxboard{G,B}$ according to the above rules results in legal positions.

\begin{example}
	Consider \col played on $P_5$. Since pieces may not be placed on the same space, or pieces by the same player placed side by side, the facets of $\auxboard{\Col, P_5}$ then consist of the edges between $x_i$ and $y_i$, between $x_i$ and $x_{i+1}$, and between $y_i$ and $y_{i+1}$. It is given in Figure \ref{fig:ColP5Aux}.
	
\begin{figure}[!ht]
	\begin{center}
	\begin{tikzpicture}[scale=2]
		\node (1) at (0,0) {$x_1$};
		\node (2) at (1,0) {$x_2$};
		\node (3) at (2,0) {$x_3$};
		\node (4) at (3,0) {$x_4$};
		\node (5) at (4,0) {$x_5$};
		\node (6) at (0,-1) {$y_1$};
		\node (7) at (1,-1) {$y_2$};
		\node (8) at (2,-1) {$y_3$};
		\node (9) at (3,-1) {$y_4$};
		\node (10) at (4,-1) {$y_5$};
		\draw (1)--(2);
		\draw (2)--(3);
		\draw (3)--(4);
		\draw (4)--(5);
		\draw (5)--(10);
		\draw (6)--(7);
		\draw (7)--(8);
		\draw (8)--(9);
		\draw (9)--(10);
		\draw (6)--(1);
		\draw (2)--(7);
		\draw (3)--(8);
		\draw (4)--(9);
	\end{tikzpicture}
	\end{center}
	\caption{The Illegal Complex $\auxboard{\Col, P_5}$}
	\label{fig:ColP5Aux}
\end{figure}

Playing on the vertices $x_1, y_3, x_4, y_5$ is legal since we never have both vertices of an edge. This position is shown on the top of Figure \ref{fig:SRCcol}, while the bottom shows the corresponding position played on $P_5$.

\begin{figure}[!ht]
	\begin{center}
	\begin{tikzpicture}[scale=1]
		\node[circle, draw, minimum size=8mm] (1) at (0,0) {$L$};
		\node[circle, draw, minimum size=8mm] (2) at (2,0) {};
		\node[circle, draw, minimum size=8mm] (3) at (4,0) {};
		\node[circle, draw, minimum size=8mm] (4) at (6,0) {$L$};
		\node[circle, draw, minimum size=8mm] (5) at (8,0) {};
		\node[circle, draw, minimum size=8mm] (6) at (0,-2) {};
		\node[circle, draw, minimum size=8mm] (7) at (2,-2) {};
		\node[circle, draw, minimum size=8mm] (8) at (4,-2) {$R$};
		\node[circle, draw, minimum size=8mm] (9) at (6,-2) {};
		\node[circle, draw, minimum size=8mm] (10) at (8,-2) {$R$};
		\node  at (0,0.75) {$x_1$};
		\node  at (2,0.75) {$x_2$};
		\node  at (4,0.75) {$x_3$};
		\node  at (6,0.75) {$x_4$};
		\node  at (8,0.75) {$x_5$};
		\node  at (0,-2.75) {$y_1$};
		\node  at (2,-2.75) {$y_2$};
		\node  at (4,-2.75) {$y_3$};
		\node  at (6,-2.75) {$y_4$};
		\node  at (8,-2.75) {$y_5$};
		\draw (1)--(2);
		\draw (2)--(3);
		\draw (3)--(4);
		\draw (4)--(5);
		\draw (5)--(10);
		\draw (6)--(7);
		\draw (7)--(8);
		\draw (8)--(9);
		\draw (9)--(10);
		\draw (6)--(1);
		\draw (2)--(7);
		\draw (3)--(8);
		\draw (4)--(9);
		\node[circle, draw, minimum size=8mm] (a) at (0,-4) {$L$};
		\node[circle, draw, minimum size=8mm] (b) at (2,-4) {};
		\node[circle, draw, minimum size=8mm] (c) at (4,-4) {$R$};
		\node[circle, draw, minimum size=8mm] (d) at (6,-4) {$L$};
		\node[circle, draw, minimum size=8mm] (e) at (8,-4) {$R$};
		\draw (a)--(b);
		\draw (b)--(c);
		\draw (c)--(d);
		\draw (d)--(e);
		\node  at (0,-4.75) {$1$};
		\node  at (2,-4.75) {$2$};
		\node  at (4,-4.75) {$3$};
		\node  at (6,-4.75) {$4$};
		\node  at (8,-4.75) {$5$};
	\end{tikzpicture}
	\end{center}
	\caption{A Legal Position on $\auxboard{\Col, P_5}$ and on $P_5$}
	\label{fig:SRCcol}
\end{figure}
\end{example}

The next example of an illegal complex has a facet of cardinality 3.
\begin{example}
Consider \nogo played on the path $P_3$. The legal ideal is
\[\legalI{\Nogo,P_3}=\langle x_1x_2, x_1x_3, x_1y_3, x_2x_3, y_1x_3, y_1y_2, y_1y_3, y_2y_3\rangle\]
while the illegal ideal is
\[\illegalI{\Nogo, P_3}=\langle x_1x_2x_3, y_1y_2y_3, x_1y_1, x_1y_2, x_2y_2, x_2y_3, x_3y_3, y_1x_2, y_2x_3\rangle.\]

The illegal complex is given in Figure \ref{fig:NogoP3aux}.

\begin{figure}[!ht]
\begin{center}
\begin{tikzpicture}
	\filldraw[fill=gray!70, draw=black] (0, 1)--(1.73, 0)--(0, -1)-- cycle;
	\filldraw[fill=gray!70, draw=black] (5.46, 1)--(3.73,0)--(5.46, -1)-- cycle;
	\draw (0,1)--(5.46,1)--(1.73,0)--(5.46,-1)--(0,-1)--(3.73,0)--(0,1);
	\draw (1.73,0)--(3.73,0);
	\filldraw (0,1) circle (0.1cm);
	\filldraw (1.73,0) circle (0.1cm);
	\filldraw (0,-1) circle (0.1cm);
	\filldraw (5.46,1) circle (0.1cm);
	\filldraw (3.73,0) circle (0.1cm);
	\filldraw (5.46,-1) circle (0.1cm);
	\draw (-0.5,1) node {$x_1$};
	\draw (1.2,0) node {$x_2$};
	\draw (-0.5,-1) node {$x_3$};
	\draw (5.96,1) node {$y_1$};
	\draw (4.26,0) node {$y_2$};
	\draw (5.96,-1) node {$y_3$};
\end{tikzpicture}
\end{center}
\caption{The Illegal Complex $\auxboard{\Nogo, P_3}$}
\label{fig:NogoP3aux}
\end{figure}
\noindent Then playing on $x_1$ and $x_2$ is legal (they form a face, but not a facet), while playing on $x_1, x_2$, and $x_3$ is illegal.
\end{example}

Similarly, playing on the legal complex $\Delta_{G,B}$ according to the following rules is also equivalent to playing $G$ on $B$:
\begin{ruleset2}\label{rules:legal}
\begin{enumerate}
	\item Left may only play on vertices labelled $x_i$, while Right may only play on vertices labelled $y_i$.
	\item The set of occupied vertices needs to be a face of $\Delta_{G,B}$.
\end{enumerate}
\end{ruleset2}

\begin{example}
Consider \col played on $C_3$. The position on the left in Figure \ref{fig:gclegal} is legal, while the one on the right is illegal when playing on the complex.

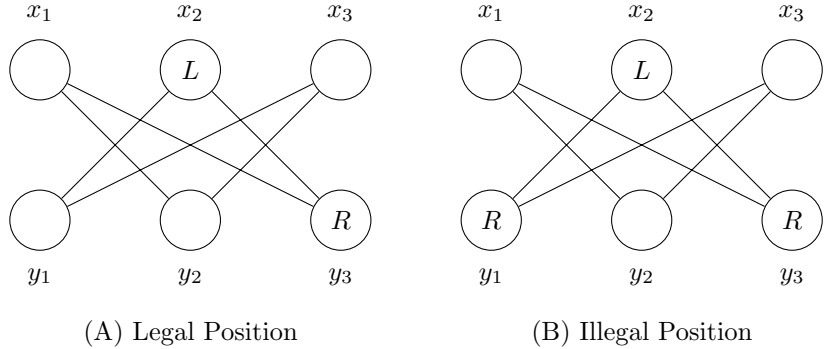
\begin{figure}[!ht]
	\begin{center}
	\begin{tikzpicture}[scale=1]
		\node[circle, draw, minimum size=8mm] (1) at (0,2) {};
		\node[circle, draw, minimum size=8mm] (2) at (2,2) {$L$};
		\node[circle, draw, minimum size=8mm] (3) at (4,2) {};
		\node[circle, draw, minimum size=8mm] (4) at (0,0) {};
		\node[circle, draw, minimum size=8mm] (5) at (2,0) {};
		\node[circle, draw, minimum size=8mm] (6) at (4,0) {$R$};
		\draw (4)--(2)--(6)--(1)--(5)--(3)--(4);
		\draw (0,-0.75) node {$y_1$};
		\draw (2,-0.75) node {$y_2$};
		\draw (4,-0.75) node {$y_3$};
		\draw (0,2.75) node {$x_1$};
		\draw (2,2.75) node {$x_2$};
		\draw (4,2.75) node {$x_3$};
		\node[circle, draw, minimum size=8mm] (a) at (6,2) {};
		\node[circle, draw, minimum size=8mm] (b) at (8,2) {$L$};
		\node[circle, draw, minimum size=8mm] (c) at (10,2) {};
		\node[circle, draw, minimum size=8mm] (d) at (6,0) {$R$};
		\node[circle, draw, minimum size=8mm] (e) at (8,0) {};
		\node[circle, draw, minimum size=8mm] (f) at (10,0) {$R$};
		\draw (d)--(b)--(f)--(a)--(e)--(c)--(d);
		\draw (6,-0.75) node {$y_1$};
		\draw (8,-0.75) node {$y_2$};
		\draw (10,-0.75) node {$y_3$};
		\draw (6,2.75) node {$x_1$};
		\draw (8,2.75) node {$x_2$};
		\draw (10,2.75) node {$x_3$};
		\node at (2,-1.5) {(A) Legal Position};
		\node at (8,-1.5) {(B) Illegal Position};
	\end{tikzpicture}
	\end{center}
	\caption{A Legal and an Illegal Position when Playing on $\Delta_{\Col, C_3}$}
	\label{fig:gclegal}
\end{figure}
\end{example}

Notice that both the legal complex and the illegal complex give a representation of the game \emph{and} the board. Thus, we can use the two complexes interchangeably, which is of advantage since sometimes the illegal complex is simpler than the legal complex (for example, the legal complex of \col played on $P_5$ has facets with 5 vertices, while in the illegal complex the facets have 2 vertices).

The next theorem recapitulates these discussions.
\begin{theorem}
Given a placement game $G$ played on a board $B$, there exist simplicial complexes $\Delta$ and $\Gamma$ such that $G$ is equivalent to the game with the Illegal Ruleset played on $\Gamma$, and equivalent to the game with the Legal Ruleset played on $\Delta$.
\end{theorem}
\begin{proof}
As shown above, $\Delta=\Delta_{G,B}$ the legal complex and $\Gamma=\Gamma_{G,B}$ the illegal complex satisfy this.
\end{proof}

\section{Discussion}
From the construction of legal complexes from placement games, there are several questions that arise naturally.

One question of interest is a possible reverse construction. In other words, we are looking at what conditions a simplicial complex has to satisfy to be a legal complex. In \cite{GameCompII} we explore this question further.

Another natural direction to pursue is how the algebra of a square-free monomial ideal $I$ (such as Cohen-Macaulayness, localization/deletion-contraction) affects the rulesets of the games played on the simplicial complexes $\facetsc{}$ and $\SRsc{}$.

\end{document}